\documentclass{siamltex}[10pt]
\usepackage{color,soul}
\usepackage{amsfonts,amsmath,amssymb,color,verbatim}
\usepackage{stmaryrd}
\usepackage{hyperref}
\usepackage[ruled,vlined]{algorithm2e}
\usepackage{tabularx,ragged2e,booktabs,caption}
\usepackage{subcaption}
\usepackage{accents}
\usepackage{pgfplots}
\usepackage{ifoddpage}
\usepackage{marginnote}
\usepackage{float} 
\usepackage{placeins} 
\usepackage[colorinlistoftodos,textwidth=4cm,shadow]{todonotes}
\usepackage{wasysym}
\usepackage{array}
\usepackage{booktabs}
\usepackage{pgf}
\usepackage{multirow}

\def\ctg {\ensuremath{\mathrm{ctg}}}

\newtheorem{remark}{Remark}[section]

\title{Convergence analysis of projected fixed-point iteration on a low-rank matrix manifold \footnotemark[1]}
\author{D. A. Kolesnikov \footnotemark[3] \and I. V.~Oseledets\footnotemark[3]\ \footnotemark[5] }

\begin{document}
\maketitle

\renewcommand{\thefootnote}{\fnsymbol{footnote}}
\footnotetext[1]{This work was supported by Russian Science Foundation grant 14-1100659}
\footnotetext[3]{Skolkovo Institute of Science and Technology,
Novaya St.~100, Skolkovo, Odintsovsky district, 143025 
Moscow Region, Russia (denis.kolesnikov@skoltech.ru, i.oseledets@skoltech.ru)}
\footnotetext[5]{Institute of Numerical Mathematics,
Gubkina St. 8, 119333 Moscow, Russia}

\renewcommand{\thefootnote}{\arabic{footnote}}
\newcounter{Ivan}
\newcommand{\Ivan}[1]{
	\refstepcounter{Ivan}{
		\todo[inline,color={green!11!blue!22},size=\small]{\textbf{[IVAN\theIvan]:}~#1}
	}
}

\newcounter{Denis}
\newcommand{\Denis}[1]{
	\refstepcounter{Denis}{
		\todo[inline,color={green!11!green!22},size=\small]{\textbf{[Denis\theDenis]:}~#1}
	}
}

\newcommand{\hlc}[2][yellow]{ {\sethlcolor{#1} \hl{#2}} }

\newcolumntype{C}[1]{>{\Centering}m{#1}}
\renewcommand\tabularxcolumn[1]{C{#1}}
\def\Twx {\ensuremath{T(\widehat{x})}}
\begin{abstract} 
In this paper we analyse convergence of projected fixed-point iteration
on a Riemannian manifold of matrices with fixed rank.
As a retraction method we use ``projector splitting scheme''.
We prove that the projector splitting scheme converges at least with the same
rate as standard fixed-point iteration without rank
constraints.
We also provide counter-example to the case when
conditions of the theorem do not hold.
Finally we support our theoretical results with numerical experiments.
\end{abstract}





\section{Introduction}
In many applications it is well-known that the solution of the optimization problem can be approximated by low-rank matrices or tensors, i.e. it lies 
on a certain manifold \cite{oseledets-survey-2015, absil-opt-2009}. Thus, instead of minimizing the full functional, the framework of Riemannian optimization can be very effective in terms 
of storage \cite{udriste-riemannian-1994, ma-manifold-2011}. There are different approaches for the optimization over low-rank manifolds, including projection onto the tangent space \cite{lubich-timett-2015}
conjugate-gradient type methods \cite{sato-cg-2015}, second-order methods \cite{absil-newton-2009}. The manifolds of matrices with bounded ranks 
and tensors with fixed tensor train and hierarchical ranks are of crucial importance in many high-dimensional problems, and are examples 
of Riemannian manifolds with a very particular polylinear structure. In this paper we consider the two-dimensional (matrix) case and study 
the convergence of the projected gradient-type methods and show that if the original method converges, its manifold version based
on the so-called \emph{projector-splitting method} is guaranteed to converge at least with the same rate and some additional conditions on the initial 
approximation.  This is up to a certain extent an unexpected result, since the standard estimates include the curvature of the manifold. 
For the manifold of matrices of rank $r$, the curvature is given by
$1 / \sigma_{\min}$, i.e. if the matrix is close to the matrix of a
smaller rank, such estimates are useless in practice. Our results show
that the curvature is not important for the convergence. 

Consider an iterative process 
\begin{equation}\label{thm:fullit}
    X_{k+1} = \Phi(X_k), \quad k = 0, \ldots
   \end{equation}
where $Y_k \in \mathbb{R}^{n \times m}$ and $\Phi$ is a contraction with parameter $\delta$. Then, 
$X_k$ converges linearly to $X_*$, for $k \rightarrow \infty$, i.e. 
$$
   \Vert X_{k+1} - X_* \Vert \leq \delta \Vert X_k - X_* \Vert,  
$$
for some matrix norm $\Vert \cdot \Vert$.
Also we assume that the initial point and the final points are on the manifold, i.e. 
and 
$$
X_0, X_* \in \mathcal{M}_r, \quad \mathcal{M}_r = \left\{ X \bigg| \rank  X \leq r \right\}.
$$

From \eqref{thm:fullit} we create the 
projected version as
\begin{equation}\label{thm:projit}
Y_{k+1} =  I(Y_k, \Phi(Y_k) - Y_k), \quad k = 0, \ldots, 
\end{equation}
where $I(Z, H)$ is the \emph{projector-splitting integrator} \cite{lubich-timett-2015} which is known to be a \emph{retraction} to the manifold \cite{absil-newton-2009}.
There are many other possible choices for the retraction, but in this paper we consider only one of them and all the convergence estimates
are proven for the method \eqref{thm:projit}.

Our approach is based on the splitting the error $\Vert X_k - X_* \Vert$ into two components. The first component is
a projection on the tangent space of the manifold at some
intermediate point and shows how close current point to stationary
point in the sense of Riemannian metric on the manifold. The second
component is the projection on normal space at the same point and is
related to the manifold curvature. The typical case convergence is
presented at Figure \ref{fig:figure1a}. However, much more interesting pattern
is possible. See Figure \ref{fig:figure1b}.

\begin{figure}[H]
\centering
\begin{subfigure}{.45\textwidth}
\centering
\scalebox{0.38}{\input{prj-est01.pgf}}
\caption{Typical case convergence.}\label{fig:figure1a}
\end{subfigure}
\hspace{0.5cm}
\begin{subfigure}{.45\textwidth}
\centering
\scalebox{0.38}{\input{prj-est02.pgf}}
\caption{Stair case convergence.}\label{fig:figure1b}
\end{subfigure}
\end{figure}

In both cases, although the curvature influences only on but the
convergence is not worse than for the full case.

\section{Projector-splitting integrator}
The projector-splitting integrator was ori\-ginally proposed \cite{lubich-lrappr-2007} as an integration scheme for the equations of motions of dynamical low-rank approximation. However,
the only information it requires, are two matrices, $A_0$, $A_1$, at subsequent time steps. Thus it is very natural to consider it for the discrete time problems, 
and moreover, it can be formally viewed as a retraction onto the manifold of rank-$r$ matrices. It is formulated as follows.

Given a rank-$r$ matrix in the form $A_0 = U_0 S_0 V_0^{\top}, U_0^{\top} U_0 = V_0^{\top} V_0 = I_r$ and a direction $D$, it provides the retraction of $A_0 + D$ back onto the manifold by the
following steps:


\begin{algorithm}[!h]
 \KwData{$A_0 = U_0 S_0 V_0^{\top}, \quad D$}
 \KwResult{$A_1 = U_1 S_1 V_1^{\top}$}
$ U_1, S' = \mathrm{QR}(U_0 S_0 + D V_0)$\;
$S'' = S' - U_1^{\top} D V^{\top}_0$\;
$V_1, S^{\top}_1 = \mathrm{QR}(V_0 S''^{\top} + D^{\top} U_1)$\;
\caption{The projector splitting retraction}\label{thm:algi}
\end{algorithm}

Note that the QR-factorizations in the intermediate steps are non-unique, but the final result $U_1 S_1 V^{\top}_1$ does not depend on it. For the details we refer the reader
to \cite{lubich-timett-2015}. We will denote the result of Algorithm
\ref{thm:algi} as $I(A_0, D)$.
Define $\mathcal{T} (X)$ as the tangent space of $X \in \mathcal{M}_r$
The following Lemma provides a new interpretation of the projector-splitting integrator as a projection onto the tangent plane in some intermediate point.
\begin{lemma}
      Let $\rank A_0 = U_0 S_0 V^{\top}_0, \quad U^{\top}_0 U_0 = V^{\top}_0 V_0 = I_r$, $D \in \mathbb{R}^{n \times m}$. Then, 
      \begin{equation}\label{thm:projectorform}
         I(A_0, D) = P_{\mathcal{T} (X)}(A_0 + D), \quad I(A_0, D),
         A_0 \in \mathcal{T} (X).
     \end{equation}
      where $X$ is some matrix of rank $r$. 
\end{lemma}
\begin{proof}
    It is sufficient to select $X = U_1 S V^{\top}_0$ for any
    non-singular $S$, and $U_1$ is defined as in the Algorithm
    \eqref{thm:algi}.
Note from the construction, that both the initial and the final points lie in the tangent space $\mathcal{T}(X)$.
\end{proof}

\section{Decomposition of the error into the normal and tangent parts}
Let us write one step of the iterative process \eqref{thm:projit} 
as 
\begin{equation}\label{thm:onestep}
   Y_1 = I(Y_0, \Phi(Y_0) - Y_0).
   \end{equation}
Using the projector form \eqref{thm:projectorform} we have
$$
Y_1 = P_{\mathcal{T}(X)}(\Phi(Y_0)),
$$
and the error can be written as 
\begin{equation}\label{thm:errorcontrol}
   E_1 = Y_1 - X_* = P_{\mathcal{T}(X)}(\Phi(Y_0) - \Phi(X_*)) + P_{\mathcal{T}(X)}(X_*) - X_*. 
   \end{equation}
Due to the contraction property we can bound 
$$
   \Vert \Phi(Y_0) - \Phi(X_*) \Vert \leq \delta \Vert E_0 \Vert.
$$
It is natural to introduce the notation
$$
P_{\mathcal{T}(X)}(X_*) - X_* = -P^{\perp}_{\mathcal{T}(X)}(X_*),
$$
since it is the normal to the tangent space component of $X_*$ at point $X$.
Thus the error at the next step satisfies 
$$
  \varepsilon^2_1 = \Vert E_1 \Vert^2 = \varepsilon^2_{\tau} + \varepsilon^2_{\perp}. 
$$
From the definition it is easy to see that
$$
   \varepsilon_{\tau} = \Vert P_{\mathcal{T}(X)}(\Phi(Y_0) - \Phi(X_*)) \Vert \leq \Vert \Phi(Y_0) - \Phi(X_*) \Vert \leq \delta \varepsilon_0. 
$$
The estimate for the decay of $\varepsilon_{\perp} = \Vert P^{\perp}_{\mathcal{T}(X)}(X_*)\Vert$ is much less trivial.
\section{Estimate for the normal component of the error}
From the definition of the error we have
$$
   \Phi(Y_0) = X_* + H, 
$$
and $\Vert H \Vert \leq \delta \varepsilon_0$. 
Since $Y$ and $X_*$ are on the manifold, they admit factorizations
$$
   Y = U_0 S_0 V^{\top}_0, \quad X_* = U_* S_* V^{\top}_*,
$$
where $U_*, V_*, U_0$ and $V_0$ are orthonormal.
If $\varepsilon_0$ is small, one can expect that the subspaces spanned by columns of $V_0$ and $V_*$ are close; however, the estimates depend 
on the smallest singular values of $X_*$. 
The following Theorem gives a bound on the normal component. 

\begin{theorem}\label{thm:normalcomponent}
    Let $X_* = U_* S_* V^{\top}_*$, where $V_*^{\top} V_* = U_*^{\top}
    U_* = I_q, \quad q \leq r$ and $H$ is an $n \times m$ matrix, $V_0$ be an $m \times r$ matrix with orthonormal 
    columns and $U_1$ be any orthogonal basis for the column space of the matrix $(X_* + H)V_0.$  Then, the norm of $P^{\perp}(X_*)$ defined as 
    \begin{equation}\label{thm:prjcomp}
       P^{\perp}(X_*) = (I - U_1 U^{\top}_1) X_* (I - V_0 V_0^{\top}). 
    \end{equation}
    can be bounded as
    \begin{equation}\label{thm:norm_normal}
        \Vert P^{\perp}(X_*) \Vert \leq \Vert H \Vert \Vert \tan \angle (V_0, V_*) \Vert. 
\end{equation}
\end{theorem}
\begin{proof}
 First, we find an $r \times r$ orthonormal matrix $Q$ such that
    \begin{equation}\label{thm:psidef}
        \Psi Q = (V^{\top}_* V_0) Q = \begin{bmatrix} \widehat{\Psi} & 0_{r-q} \end{bmatrix}, 
    \end{equation}
    where matrix $\widehat{\Psi}$ has size $q \times q$.
    Since the multiplication by the orthogonal matrix $Q$ does not change the projector 
    $$V_0 V^{\top}_0 = (V_0 Q) (V_0 Q)^{\top},$$
    we can always assume that the matrix $\Psi$ is already in the  form \eqref{thm:psidef}.
    Since $U_1$ spans the columns space of $(X_* + H) V_0$, we have
    \begin{equation}\label{thm:singvecteq}
(U_1 U^{\top}_1) (X_* V_0  + H V_0) = X_* V_0 + H V_0.
\end{equation}
From this equation we have
\begin{equation}\label{thm:eq1}
X_* V_0 = U_1 U^{\top}_1 X_* V_0 + U_1 U^{\top}_1 H V_0 - H V_0 = U_* S_* V^{\top}_* V_0 = U_* \begin{bmatrix} \widehat{\Psi} & 0 \end{bmatrix}.
\end{equation}
Introduce the matrix $V^{(q)}_0$ comprised of the first $q$ column of the matrix $V_0$. From \eqref{thm:eq1} we have
$$
U_* S_* \widehat{\Psi} = U_1 U^{\top}_1 X_* V^{(q)}_0 + U_1 U^{\top} H V^{(q)}_0 - H V^{(q)}_0. 
$$
Thus, 
\begin{equation}\label{thm:ustar}
U_* S_*= U_1 \Psi_1  - H V^{(q)}_0 \widehat{\Psi}^{-1},
\end{equation}
Note, that 
$$\Vert P^{\perp} \Vert = \Vert (I - U_1 U^{\top}_1) X_* (I - V_0 V^{\top}_0) \Vert \leq \Vert (I - U_1 U^{\top}_1) X_* 
(I - V^{(q)}_0  V^{(q)}_0)^{\top})) \Vert,$$
and from \eqref{thm:singvecteq} it follows also that 
$$
(I - U_1 U^{\top}_1) (X_* + H) V^{(q)}_0 (V^{(q)}_0)^{\top} = 0.
$$
For simplicity, denote
$$
P^{\perp}_q(X_*) = (I - U_1 U^{\top}_1) X_* 
(I - V^{(q)}_0 (V^{(q)}_0)^{\top}).
$$
Then, 
    \begin{equation}\label{thm:perpproj}
        \begin{split}
            P_q^{\perp}(X_*) & = (I - U_1 U^{\top}_1) X_* - (I  - U_1 U^{\top}_1) X_* V^{(q)}_0 (V^{(q)}_0)^{\top} = \\ &= (I - U_1 U^{\top}_1) X_* + 
            (I - U_1 U^{\top}_1) H V^{(q)}_0 (V^{(q)}_0)^{\top}.
  \end{split}
  \end{equation}
Replacing $U_* S_*$ in \eqref{thm:perpproj} by \eqref{thm:ustar} we get  
\begin{equation}
    \begin{split}
        P^{\perp}_{\mathcal{T}(X)} &= (I - U_1 U_1^{\top}) U_* V^{\top}_* + (I - U_1 U^{\top}_1) H V^{(q)}_0 (V^{(q)}_0)^{\top}) \\
                                   &= (I - U_1 U^{\top}_1) H V^{(q)}_0 (V^{(q)}_0)^{\top}  - (I - U_1 U^{\top}_1) H V^{(q)}_0 \widehat{\Psi}^{-1} V^{\top}_* \\
                                   &= (I - U_1 U^{\top}_1) H V^{(q)}_0 (V^{(q)}_0)^{\top} - \widehat{\Psi}^{-1} V^{\top}_*).
\end{split}
\end{equation}
To estimate the norm, note that
$$
\Vert P^{\perp}_{\mathcal{T}(X)} \Vert \leq \Vert H \Vert \Vert (V^{(q)}_0)^{\top} -  \widehat{\Psi}^{-1} V^{\top}_* \Vert.
$$
Introduce the matrix 
$$
B = (V^{(q)}_0)^{\top} -  \widehat{\Psi}^{-1} V^{\top}_*. 
$$
We have
$$
\Vert X_* ( I - V_0 V^{\top}_0) \Vert =  \Vert (X_* - Y_0) ( I - V_0 V^{\top}_0) \Vert \leq \Vert X_* - Y_0 \Vert.
$$
Replacing $X_*$ by $U_* S_* V^{\top}_*$ we have
$$
\Vert U_* (V^{\top}_* - \Psi V^{\top}_0) \Vert = \Vert U_* (V^{\top}_* - \widehat{\Psi} (V^{(q)}_0)^{\top}) \Vert.
$$
Thus, 
$$\Vert V^{\top}_* -  \widehat{\Psi} (V^{(q)}_0)^{\top} \Vert \leq
\frac{\Vert X_* - Y_0 \Vert}{\sigma_q}.$$
Introduce the matrix $C = V^{\top}_* -  \widehat{\Psi} (V^{(q)}_0)^{\top}$.
Then, 
$$\Vert C \Vert^2 = \Vert C C^{\top} \Vert = \Vert I - \widehat{\Psi} \widehat{\Psi}^{\top} \Vert \leq \frac{\Vert X_* - Y_0 \Vert^2}{\sigma^2_q}.$$
Then, we have
$$\sin \theta \leq \frac{\Vert X_* - Y_0 \Vert}{\sigma_q},$$
whereas we require to bound
$$
\tan \theta = \frac{\sin \theta}{\sqrt{1 - \sin^2 \theta}}.
$$
%
Let $\widehat{\Psi} = U \Lambda V^{\top}$ be the singular value decomposition of $\widehat{\Psi}$. 
From the definition of the angles between subspaces we have
$$
   \Lambda = \cos \angle (V^{\top}_*, V^{(q)}_0) = \cos \angle (V^{\top}_*, V_0),
$$
therefore
$$
\Vert B \Vert^2 = \Vert \cos^{-2} \angle (V^{\top}_*, V_0) - 1\Vert = \Vert \tan^2 \angle  (V^{\top}_*, V_0) \Vert,
$$
which completes the proof.
\end{proof}

\section{Error estimate}
Theorem \ref{thm:normalcomponent} shows that the normal component can
decay as a tangent component squared. Unfortunately, convergence of the
projector splitting method in general is not guaranteed. In section
\ref{prj:counter_example_section} we give the example for which
sequence $Y_k$ converges to a matrix different from $X_*$. In this
section we derive sufficient conditions for convergence of projector splitting method.

We consider one step of the projector splitting scheme.
\begin{lemma}\label{prj:y_lem} Let us denote the initial point $Y_0 =  U_0 S_0 V_0^{\top}$,
  the next step point $  Y_1 = U_1 S_1 V_1^{\top}$ and the fixed point
  $X_* = U_* S_* V_*^{\top}.$ We assume that $S_*$ is a diagonal matrix:
  $$ S_* = \sum_{k=1}^r\limits s_k e_k e_k^{\top},$$
  where $s_k$ is the $k$-singular value and $e_k$  is the
  corresponding vector from the standard
  basis.
Let us denote 
\begin{align*}\label{prj:cos_def}
&\cos^2 \phi_{Li, k} = \Vert U_i U_i^{\top} U_* e_k \Vert_F^2, 
&\cos^2 \phi_{Ri, k} = \Vert e_k^{\top} V_*^{\top} V_i V_i^{\top}
\Vert_F^2,\\
&\sin^2 \phi_{Li, k} = \Vert (I-U_i U_i^{\top}) U_* e_k \Vert_F^2,
&\sin^2 \phi_{Ri, k} = \Vert e_k^{\top} V_*^{\top} (I-V_i V_i^{\top}) \Vert_F^2.
\end{align*}
Assume that

\begin{equation}\label{prj:pq_cond}
\begin{split}
\delta^2 \Vert Y_0 - X_*\Vert_F^2 + \sum_{k=1}^r\limits s_k^2 \sin^2
\phi_{R0, k} \leq s_r .
\end{split}
\end{equation}
Then the next inequality holds:
\begin{equation}\label{prj:y_ineq}
\begin{split}
&\Vert Y_1 - X_* \Vert_F^2 \leq \delta^2 \Vert Y_0 - X_*
\Vert_F^2 +\\
+ \Big(\delta^2 &\Vert Y_0 - X_*
\Vert_F^2 - \sum_{k=1}^r\limits s_k^2 \sin^2 \phi_{R1, k} \Big) 
 \frac{ \sum_{k=1}^r\limits s_k^2\sin^2 \phi_{R0, k}}{ s_r - \sum_{k=1}^r\limits s_k^2\sin^2
\phi_{R0, k} - \sum_{k=1}^r\limits s_k^2 \sin^2 \phi_{R1, k}}.
\end{split}
\end{equation}
\end{lemma}
\begin{proof}
Without the loss of generality we can assume that

$$
U_1 = \begin{bmatrix} I_r \\      0_{(n-r)\times r}\\ \end{bmatrix}, 
\quad
V_0 = \begin{bmatrix}      I_r \\ 0_{(m-r)\times r}\\\end{bmatrix}.
$$
Then we use the following block representation of $Y_0, \Phi(Y_0), Y_1$ and $X_*$:

\begin{equation*}
\begin{split}
Y_0 &= U_0 S_0 V_0^{\top} =
 \begin{bmatrix}
         D_1^0 & 0 \\
         D_3^0 & 0
       \end{bmatrix}, 
\Phi(Y_0) = 
 \begin{bmatrix}
          D_1^1 & D_2^2 \\
          0 & D_4^1 \\
   \end{bmatrix},\\
Y_1 &= U_1 S_1 V_1^{\top} = 
 \begin{bmatrix}
           D_1^1 & D_2^1 \\
           0 & 0 \\
   \end{bmatrix},
X_* = U_* S_* V_*^{\top} = 
 \begin{bmatrix}
           E_1 & E_2 \\
           E_3 & E_4 \\
   \end{bmatrix}.
\end{split}
\end{equation*}
Therefore,

\begin{equation*}
\begin{split}
\Vert Y_1 - X_* \Vert_F^2 &= \Vert D_1^1 - E_1 \Vert_F^2 + \Vert D_2^1
- E_2 \Vert_F^2 + \Vert E_3 \Vert_F^2 + \Vert E_4 \Vert_F^2 \leq\\
&\leq \left(\Vert D_1^1 - E_1 \Vert_F^2 + \Vert D_2^1
- E_2 \Vert_F^2 + \Vert E_3 \Vert_F^2 + \Vert D_4^1 - E_4
\Vert_F^2\right) + \left( \Vert E_4 \Vert_F^2\right) =\\
&=\Vert \Phi(Y_0) - X_* \Vert_F^2 + \Vert (I - U_1 U_1^{\top})  X_* (I - V_0 V_0^{\top}) \Vert_F^2 \leq \\
&\leq \delta^2 \Vert Y_0 - X_*
\Vert_F^2 + \Vert (I - U_1 U_1^{\top})  X_* (I - V_0 V_0^{\top})\Vert_F^2.
\end{split}
\end{equation*}
We want to estimate $ \Vert (I - U_1 U_1)  X_* (I - V_0
V_0)\Vert_F^2$. For that purpose we exploit contraction property of
$\Phi$:

\begin{equation*}\label{prj:y_left_ineq}
\begin{split}
\Vert U_1 U_1^{\top} ( \Phi(Y_0) - X_*)\Vert_F^2 &+ \Vert (I - U_1U_1^{\top} ) ( \Phi(Y_0) - X_*)\Vert_F^2 =\\
&= \Vert( \Phi(Y_0) - X_*)\Vert_F^2 \leq \delta^2 \Vert Y_0 - X_* \Vert_F^2, \\
\Vert U_1 U_1^{\top} ( X_*) (I - V_1 V_1^{\top})\Vert_F^2 &+ 
\Vert (I - U_1U_1^{\top} ) ( X_*)V_0 V_0^{\top}\Vert_F^2 \leq\delta^2 \Vert Y_0 - X_* \Vert_F^2,\\
\Vert (I - U_1U_1^{\top} ) ( X_*)  V_0 V_0^{\top}\Vert_F^2
&-\Vert (I - U_1 U_1^{\top}) ( X_*) (I - V_1 V_1^{\top})\Vert_F^2
 \leq \\
&\leq \delta^2 \Vert
Y_0 - X_* \Vert_F^2 -  \Vert ( X_*) (I - V_1 V_1^{\top})\Vert_F^2.
\end{split}
\end{equation*}
Then the inequality \eqref{prj:y_left_ineq} transforms to

\begin{equation}\label{prj:left_ineq1}
\begin{split}
&\sum_{k=1}^r s_k^2 \Vert (I - U_1 U_1^{\top} ) U_* e_1 \Vert_F^2
\Vert e_k^{\top} V_0 V_0^{\top}\Vert_F^2 -\\
&- \sum_{k=1}^r s_k^2  \Vert (I - U_1 U_1^{\top})U_* e_k\Vert \Vert
e_k^{\top} V_*^{\top} (I - V_1 V_1^{\top})\Vert_F^2
 \leq \\
&\leq \delta^2 \Vert
Y_0 - X_* \Vert_F^2 -  \sum_{k=1}^r s_k^2 \Vert U_* e_k \Vert_F^2
\Vert e_k^{\top} V_*^{\top} (I - V_1 V_1^{\top})\Vert_F^2.
\end{split}
\end{equation}
Using \eqref{prj:cos_def} 
we have

$$ 
\sum_{k = 1}^r\limits \sin^2 \phi_{L1, k} s_k^2 (\cos^2
\phi_{R0, k} - \sin^2 \phi_{R1, k}) \leq \delta^2 \Vert Y_0 - X_*
\Vert_F^2 - \sum_{k=1}^r\limits s_k^2 \sin^2 \phi_{R1, k}.
$$
Inequality \eqref{prj:pq_cond} guarantees that 
\begin{equation*}
\begin{split}\sum_{k=1}^r\limits s_k^2\sin^2
\phi_{R0, k} - \sum_{k=1}^r\limits s_k^2 \sin^2 \phi_{R1, k} < s_r^2, \\
0 < \max_{1\leq k\leq r} \left(\cos^2
\phi_{R0, k} - \sin^2 \phi_{R1, k}\right).\\
\end{split}
\end{equation*}
Therefore

\begin{equation*}
\begin{split}
&\sum_{k = 1}^r\limits s_k^2 \sin^2 \phi_{L1, k} \sin^2\phi_{R0, k}  \leq \\
\leq \Big(\delta^2 \Vert Y_0 - X_*
\Vert_F^2 - &\sum_{k=1}^r\limits s_k^2 \sin^2 \phi_{R1, k} \Big) 
\max_{1\leq k\leq r} \frac{\sin^2 \phi_{R0, k}}{\cos^2
\phi_{R0, k} - \sin^2 \phi_{R1, k}} \leq \\
\leq \Big(\delta^2 \Vert Y_0 - X_*
\Vert_F^2 - &\sum_{k=1}^r\limits s_k^2 \sin^2 \phi_{R1, k} \Big) 
 \frac{ \sum_{k=1}^r\limits s_k^2\sin^2 \phi_{R0, k}}{ s_r^2 - \sum_{k=1}^r\limits s_k^2\sin^2
\phi_{R0, k} - \sum_{k=1}^r\limits s_k^2 \sin^2 \phi_{R1, k}}.
\end{split}
\end{equation*}
i.e. \eqref{prj:y_ineq} is proven.
\end{proof}\\
For convenience we introduce new variables:

\begin{equation}\label{prj:pqs}
\begin{split}
s = \delta^2, \quad p_k = \dfrac{\Vert Y_k - X_* \Vert^2_F}{s_r^2},
\quad q_k =  \dfrac{1}{s_r^2}\sum_{k=1}^r\limits s_k^2 \sin^2 \phi_{Rk},
\end{split}
\end{equation}
Now we can formulate the connection between the subsequent steps:

\begin{equation}\label{prj:pqs_ineq}
\begin{split}
p_{k+1} \leq s p_k + \frac{(s p_k- q_{k+1})q_k}{1 - q_k - q_{k+1}},
\quad 0 \leq q_{k+1} \leq s p_k. 
\end{split}
\end{equation}
We can derive upper estimate for $p_k$:

\begin{theorem}
Assume that $0<s<1$, $0 \leq q_0 \leq 1$,
  $0<p_0$. \\
Consider $p_k , q_k , k\in \mathbf{N}$ that satisfy \eqref{prj:pqs}. 
Assume that $4 \dfrac{p_0}{(1-q_0)^2}\dfrac{s}{1-s}<1$. Then the next inequalities hold:
\begin{equation}\label{prj:thm_spq}
\begin{split}
p_k &\leq\frac{p_0}{c_*(s, p_0, q_0)} s^k,\quad 0 < c_*(s, p_0, q_0)
\leq 1 - \sum_{j=0}^k \limits q_k \leq s p_{k-1} + q_{k-1},
\end{split}
\end{equation}
where
 $$ c_*(s, p_0, q_0) = \frac{p_0}{1-q_0}
\frac{s}{1-s} \left(\dfrac{2}{1 + \sqrt{1 - 4
    \frac{p_0}{(1-q_0)^2}\frac{s}{1-s}}} \right). $$

\end{theorem}
\begin{proof}
The parameter $c_* (s, p_0, q_0)$ is the
positive solution of the equation: 
$$c_* (s, p_0, q_0) = 1 - q_0 - p_0 \frac{s}{1-s}
\frac{1}{c_* (s, p_0, q_0)}.$$

We will use mathematical induction to prove \eqref{prj:thm_spq}. The
base case follows from $0 <
c_* (s, p_0, q_0) < 1 $
$$
p_0 \leq \frac{p_0}{c_* (s, p_0, q_0)}, \quad c_* (s, p_0, q_0)
\leq 1 - q_0.
$$

Consider the inductive step. Assume that \eqref{prj:thm_spq}
holds for every $i<k$ for some k. Then, 

\begin{equation}
\begin{split}
p_{k+1} &\leq s p_k + \frac{(s p_k- q_{k+1})q_k}{1 - q_k - q_{k+1}} = \\
&= s p_k \frac{1 - q_k}{1 - q_k - q_{k+1}} - \frac{ q_{k+1}q_k}{1 - q_k  - q_{k+1}}  \leq s \frac{p_k}{1 - \frac{q_{k+1}}{1 - q_k}}.
\end{split}
\end{equation}
We can expect that the term $\dfrac{ q_{k+1}q_k}{1 - q_k - q_{k+1}}$ is
sufficiently smaller than the $p_{k+1}$ and decays as $p_{k+1}^2$ due
to $q_k\sim p_k$. Finally,

\begin{equation}\label{prj:pqs_ineq2}
\begin{split}
p_{k+1} &\leq 
\dfrac{sp_k}{1 - \left(\dfrac{q_{k+1}}{1 - q_k}\right) } \leq  \dfrac{s^{k+1} p_0}{\prod\limits_{j = 0}^k \Big( 1 - \dfrac{q_{j+1}}{1 - q_j} \Big)}.
\end{split}
\end{equation}
It is easy to prove that in the case $\sum_{j = 0}^k\limits q_j < 1$ we have 

$$
\prod\limits_{j = 0}^k \Big( 1 - \frac{q_{j+1}}{1 - q_j} \Big)
\leq 1 - \sum_{j = 0}^{k+1} q_k.
$$
It leads to

\begin{equation*}
\begin{split}
p_{k+1} &\leq \dfrac{s^{k+1} p_{0}}{1 - \sum_{j = 0}^k\limits
    q_j }
\leq \dfrac{s^{k+1} p_{0}}{c_* (s, p_0, q_0)},\\
\end{split}
\end{equation*}
therefore

\begin{equation*}
\begin{split}
c_*(s, p_0, q_0) &= 1 - q_0 -  \dfrac{p_0}{c_*(s, p_0, q_0)}
\frac{s}{1-s} = 1 - q_0 - s\sum_{k=0}^{\infty}\limits \dfrac{p_0}{c_*(s, p_0, q_0)}
s^i \leq \\
&\leq 1 - q_0 - s \sum_{j=0}^{k+1}\limits p_j \leq 1 -
\sum_{j=0}^{k+1}\limits q_j \leq 1 - q_k - s p_k.\\
\end{split}
\end{equation*}
The inductive step is proven.
\end{proof}

The final estimate is

\begin{equation*}
\begin{split}
 p_n \leq \dfrac{ p_{0} }{ c_*(s,
  p_0, q_0) } s^{n} = \dfrac{ p_{0} }{ 1 - q_0 } s^{n} \left(\dfrac{1 + \sqrt{1 - 4
    \dfrac{p_0}{(1-q_0)^2}\dfrac{s}{1-s}}}{2\dfrac{p_0}{(1-q_0)^2}\dfrac{s}{1-s}}\right).
\end{split}
\end{equation*}

Note that if the condition $4 \dfrac{p_0}{(1-q_0)^2}\dfrac{s}{1-s}<1$
does hold, then the condition $s p_0 + q_0< 1$ does hold as well.

\begin{corollary}
Define $Y_k$ as in \eqref{thm:projit}, $X_*$, $s_k$ and $\sin^2 \phi_{R0, k}$ as in Lemma \ref{prj:y_lem}.
Assume that the next inequality holds
$$
 4 \dfrac{\Vert Y_0 - X_*\Vert}{\left(s_r^2 - \sum_{k=1}^r\limits s_k^2 \sin^2 \phi_{R0,
       k}\right)^2} < 1.
$$

Then the sequence $Y_k$ converges to $X_*$ and the following
inequality holds
$$
\Vert Y_k - X_*\Vert < c(\delta, Y_0, X_*) \Vert Y_0 - X_*\Vert \delta^k, 
$$
where
$$
c(\delta, Y_0, X_*) = \dfrac{1 + \sqrt{1 - 4
    \dfrac{\delta^2\Vert Y_0 - X_*\Vert^2 }{(1-\delta^2)\left(s_r^2-\sum_{k=1}^r\limits s_k^2 \sin^2 \phi_{R0,
       k}\right)^2}}}{2 \dfrac{\delta^2\Vert Y_0 - X_*\Vert^2 }{(1-\delta^2)\left(s_r^2-\sum_{k=1}^r\limits s_k^2 \sin^2 \phi_{R0,
       k}\right)^2}}.
$$
\end{corollary}
This estimate guarantees if the initial point is close enough to the fixed
point then the projector
splitting method in the worst case has the same convergence rate as
the fixed-point iteration method. Also the estimate requires that the
distance between the initial point and the fixed point is less than
the smallest singular value of the fixed point $s_r$. In the next
section we give the example for which this condition do not hold and
the projector splitting method does not converges to the true solution.

\section{Counter-example}\label{prj:counter_example_section}
Consider the case $n=2,$ $r=1$. 
We will need the following auxiliary result:

\begin{lemma}\label{prj:f_lemma}
Let the mapping $\Phi:
\mathbb{R}^{2\times 2} \to \mathbb{R}^{2\times 2}$ be defined as

\begin{equation}\label{prj:phidef}
\begin{split}
\Phi(Y) &= X_* + \delta \Vert Y - X_* \Vert_F X_{\perp}, \\
 X_* &= \begin{pmatrix} 1 & 0 \\ 0 & 0\end{pmatrix}, X_{\perp} =  \begin{pmatrix} 0 & 0 \\ 0 & 1\end{pmatrix}.\\
\end{split}
\end{equation}
Let us consider $\delta, d_*, q_{max}$ and $s$ that satisfy
\begin{equation}
\begin{split}
 0 &< \delta < 1, \quad 1< \delta^2 + \delta^6, \quad d_* = \frac{1}{\sqrt{1-\delta^2}}, \\
 0 &< q_{max}, \quad 0<s, \quad \frac{1}{\delta^4 d_*^2} \Big( 1
 +\frac{q_{max}}{\delta^2 d_*^2 } \Big) \leq \delta^2 - \frac{s}{\delta^2 d_*^2}.\\
\end{split}
\end{equation}
Denote the set 
$$\Omega = \left\{ \{p,  q\} \Big| 0\leq p,\quad 0\leq  q \leq q_{max},\quad \frac{q}{p} \leq s\right\} $$
and the function 

\begin{equation*}
\begin{split}
f: \Omega \to \mathbb{R}_{0, +}^2,\quad f(\{ p,
  q \}) = \left\{ \dfrac{1 + \delta^2 p}{1 +
    \dfrac{q}{\delta^2 d_*^2 (1 + p)}  }-1,\quad  \dfrac{q}{\delta^4
    d_*^2 (1 + p)} \right\},
\end{split}
\end{equation*}
\end{lemma}
and
$$
f^{*n}(x) = \underbrace{f(\ldots f(x) \ldots )}_{n \text{ times}}.
$$
Then 
$$
f(\Omega) \subset \Omega, \quad 
\forall x\in\Omega,\quad\lim_{n\to\infty}\limits f^{*n}(x) = \{0,  0\}.
$$

\begin{proof}
It is important to note that $1 < \delta^4 d_*^2 = \dfrac{\delta^4}{1 -
  \delta^2}$ because of the choice of $\delta$\eqref{prj:f_lemma}.
Let us denote $f(\{ p,
  q \}) = \{p_1,  q_1 \}$. Then 

\begin{equation}\label{prj:ineq_qmax}
\begin{split}
q_1 = \frac{q}{\delta^4 d_*^2 (1 + p)} \leq \frac{q}{\delta^4
    d_*^2} < q \leq q_{max},
\end{split}
\end{equation}
and therefore

\begin{equation*}
\begin{split}
\frac{q_1}{q} &= \frac{1}{\delta^4 d_*^2 (1 + p)} \leq \frac{1}{\delta^4 d_*^2},\\
\frac{p_1}{p} &= \frac{1}{p}\Big(  d\frac{1 + \delta^2 p}{1 +
    \dfrac{q}{\delta^2 d_*^2 (1 + p)}  }-1\Big) \geq 
\frac{1}{p}\Big(  \dfrac{1 + \delta^2 p}{1 +
    \dfrac{q}{\delta^2 d_*^2}  }-1\Big) =\\
&=  \Big(\delta^2 - \frac{1}{\delta^2 d_*^2}\frac{q}{p}\Big) / \Big(1 +
    \frac{q}{\delta^2 d_*^2}  \Big) \geq \Big(\delta^2 -
    \frac{s}{\delta^2 d_*^2}\Big) / \Big(1 +
    \frac{q_{max}}{\delta^2 d_*^2} \Big) \geq \frac{1}{\delta^4
      d_*^2}.\\
\end{split}
\end{equation*}
Finally we have 

\begin{equation}\label{prj:ineq_pq}
\begin{split}
\frac{q_1}{p_1} \leq \dfrac{ q / \delta^4
      d_*^2}{ p / \delta^4 d_*^2} = \frac{q}{p}\leq s.
\end{split}
\end{equation}
The statement $f(\Omega) \subset \Omega$ follows from
\eqref{prj:ineq_qmax} and \eqref{prj:ineq_pq}. Also the following
inequalities hold

\begin{equation*}\label{prj:conv_pq}
\begin{split}
\frac{p_1}{p} = \frac{1}{p}\left(  \dfrac{1 + \delta^2 p}{1 +
    \dfrac{q}{\delta^2 d_*^2 (1 + p)}  }-1\right) \leq \delta^2,\quad \frac{q_1}{q} \leq \frac{1}{\delta^4 d_*^2}.
\end{split}
\end{equation*}
The inequalities \eqref{prj:conv_pq} guarantee linear convergence of  $ f^{*n}(x)$ to $ \{ 0,  0\}$ for
every $x \in \Omega.$
\end{proof}
 
\begin{lemma}\label{prj:pi_lemma}
Let contraction mapping $\Phi$ is defined as in lemma \ref{prj:f_lemma}.
Let us consider parameters $\delta, d_*$, contraction mapping $\Phi$ and the
set $\Omega$ and the function $f$ that satisfy condition of lemma \ref{prj:f_lemma}.
Let us denote  the set of rank-$1$ $2\times2$ real matrices
 $M_{2,1}(\mathbf{R})$, 
 $\phi_R(X)$ - right angle for rank-$1$ $2\times2$ matrix $X$ and
\begin{equation}
\begin{split} 
&\mathcal{M}_{2,1}^{'} = \Big( X | X \in M_{2,1}(\mathbf{R}), \sin^2(\phi_R(X)) > 0\Big),\\
&\pi : \mathcal{M}_{2,1}^{'} \to \mathbf{R}_{0, +}^2,\quad \pi(X)= \left\{
  \dfrac{\Vert X - X_*\Vert_F^2}{d_*^2} - 1, \quad \ctg^2 \phi_R(X) \right\}.\\
\end{split}
\end{equation}
Assume that 
\begin{equation}
\begin{split}
Y_0 &=  \begin{pmatrix} \cos \phi_{L0} & \sin
  \phi_{L0}\end{pmatrix} s_0 \begin{pmatrix} \cos \phi_{R0} \\ \sin
  \phi_{R0} \end{pmatrix} \in \pi^{-1}(\Omega), \\
Y_1 &= I(Y_0, \Phi(Y_0) - Y_0) = \begin{pmatrix} \cos \phi_{L1} & \sin
  \phi_{L1}\end{pmatrix} s_0 \begin{pmatrix} \cos \phi_{R1} \\ \sin
  \phi_{R1} \end{pmatrix}.
\end{split}
\end{equation}
Then the following equalities hold
\begin{equation}\label{prj:commuteq}
\begin{split}
\pi (Y_1) = f(\pi(Y_0)),\quad
\ctg^2 \phi_{L1} <\ctg^2 \phi_{R1}. 
\end{split}
\end{equation}

\end{lemma}

\begin{proof}
We will use the equivalent form of Algorithm \ref{thm:algi}

\begin{equation}
\begin{split}
 U_1, S' &= \mathrm{QR}( (A_0+ D) V_0),\\
V_1, S^{\top}_1 &= \mathrm{QR}( (A_0+ D^{\top}) U_1).\\
\end{split}
\end{equation}
Let us consider $Y_0 = U_0 S_0 V_0^{\top}$,  $d_0 =
\Vert Y_0 - X_*\Vert_F$ and $V_0 = \begin{pmatrix} \cos \phi_{R0} \\ \sin
  \phi_{R0}\end{pmatrix}$.
Then
\begin{equation}
\begin{split}
\Phi(Y_0) &=  \begin{pmatrix} 1 & 0 \\ 0 & \delta d_0\end{pmatrix},
\quad U_1, S' = \mathrm{QR}\left(  \begin{pmatrix} \cos \phi_{R0} \\ \delta d_0 \sin \phi_{R0}\end{pmatrix}\right),\\
V_1, S^{\top}_1 &= \mathrm{QR}\left(  \frac{1}{\sqrt{1 + (\delta^2 d_0^2-1)
  \sin^2 \phi_{R0}}}\begin{pmatrix} \cos \phi_{R0} \\ \delta^2
  d_0^2 \sin \phi_{R0}\end{pmatrix} \right).\\
\end{split}
\end{equation}
Finally we get:

\begin{equation}\label{prj:usveq}
\begin{split}
U_1 S_1 V_1^{\top} =  \begin{pmatrix} \cos \phi_{R0} \\ \delta d_0 \sin
  \phi_{R0} \end{pmatrix} \frac{1}{1 + (\delta^2 d_0^2-1) \sin^2 \phi_{R0}} \begin{pmatrix} \cos \phi_{R0} & \delta^2 d_0^2 \sin
  \phi_{R0} \end{pmatrix}
\end{split}
\end{equation}
It is important to note that $\cos^2 \phi_{L1} < \cos^2 \phi_{R1} <
\cos^2 \phi_{R0} $ in case $1 < \delta d_*$ (and our choice of
$\delta$ provides that).
The equality \eqref{prj:usveq} guarantees if $0 < \sin^2 \phi_{R0}$
then $0 < \sin^2 \phi_{R1}.$ 
So 
\begin{equation}
\begin{split}
d_1^2 &= S_1^2  + \Big(1 - \frac{\cos^2 \phi_{R0}}{\cos^2 \phi_{R0} + \delta^2 d_0^2
  \sin^2 \phi_{R0}}\Big)^2 - \Big(\frac{\cos^2 \phi_{R0}}{\cos^2
  \phi_{R0} + \delta^2 d_0^2
  \sin^2 \phi_{R0}}\Big)^2 = \\
&= \frac{\cos^2 \phi_{R0} + \delta^4 d_0^4
  \sin^2 \phi_{R0} }{\cos^2 \phi_{R0} +\delta^2 d_0^2
  \sin^2 \phi_{R0}} + 1 - \dfrac{2\cos^2 \phi_{R0}}{\cos^2 \phi_{R0} + \delta^2 d_0^2
  \sin^2 \phi_{R0}} = \frac{1+ \delta^2 d_0^2}{1 + \ctg^2 \phi_{R0} /
  (\delta^2 d_0^2)}
\end{split}
\end{equation}
Let us denote $p_0 = d_0^2 / d_*^2 - 1$ and $q_0 =
\ctg^2(\phi_{R0})$. Then

\begin{equation}
\begin{split}
\frac{d_1^2}{d_*^2} - 1 = \frac{1}{d_*^2} \left(\frac{1+ \delta^2 d_0^2}{1 + \ctg^2 \phi_R /
  (\delta^2 d_0^2)}\right) - 1 = \\ \frac{1}{d_*^2} \left(\frac{1+ \delta^2 d_0^2}{1+ q_0 /
  (\delta^2 d_0^2)}\right) - 1 =  \dfrac{1 + \delta^2 p_0}{1 +
    \dfrac{q_0}{\delta^2 d_*^2 (1 + p_0)}  }-1,\\ 
q_1 = \ctg^2(\phi_{R1}) = \frac{\ctg^2(\phi_{R0})}{\delta^2 d_0^2} = \frac{q_0}{\delta^2 d_*^2(1+p_0)}. 
\end{split}
\end{equation}
It completes the proof of \eqref{prj:commuteq}. 
\end{proof}

\begin{theorem}\label{prj:counter_thm}
Let the mappings $\Phi, \pi$ and the set $\Omega$ are defined as in
lemma \ref{prj:pi_lemma}.
Let us consider matrix $Y_0 \in \pi^{-1}(\Omega)$ and the projector
splitting integrator $I(A, D)$ that is defined by \eqref{thm:algi}. Then the sequence $Y_k =
I(Y_{k-1}, \Phi(Y_{k-1})-Y_{k-1})$ converges to $Y_* = d_* X_{\perp}.$
\end{theorem}
\begin{proof}
We apply Lemma \ref{prj:pi_lemma}
$$\pi(Y_k) = f(\pi(Y_{k-1})) = f^{*k} (\pi(Y_0) ) $$
and then, using Lemma  \ref{prj:f_lemma}, we have
$$\lim_{k \to \infty} \pi(Y_k) = \lim_{k \to \infty}\limits f^{*k} (\pi(Y_0) ) =\{0, 0\}.$$
Lemma \ref{prj:pi_lemma} guarantees that squared cotangents of left
and right angles go to zero, so $\lim_{k \to \infty}\limits  Y_k = Y_*$.
\end{proof}

\begin{remark}
Note that the condition $1< \delta^2 + \delta^6$ (it requites $\delta >
0.8$) significantly restricts the usage of Theorem
\ref{prj:counter_thm}. But our numerical experiments show that the
projector splitting method might not converge in computer arithmetics
in the case this condition does not hold.
\end{remark}

\section{Numerical examples}
\subsection{Typical case}\label{prj:bunchcase}
We consider the ''linear'' contraction mapping 
\begin{equation*}
\begin{split}
\Phi: \mathbb{R}^{n\times m} \to \mathbb{R}^{n\times m}, \quad \Phi(X) = X_* + Q(X-X_*),
\end{split}
\end{equation*}
where $X$ and $ X^* $ are rank-$r$ $n\times m$ matrices, $Q$ is a linear
operator (on matrices), $n=m=40$, $r = 7$, $\Vert Q \Vert < 0.8$ 
let us denote singular values of $X_*$ as $\sigma_i, 1\leq i \leq r.$
The typical case corresponds to $\sigma_1 / \sigma_r \approx
10$. It shows that the orthogonal part converges quadratically.
\begin{figure}[H]
    \centering
    \scalebox{0.7}{
    \input{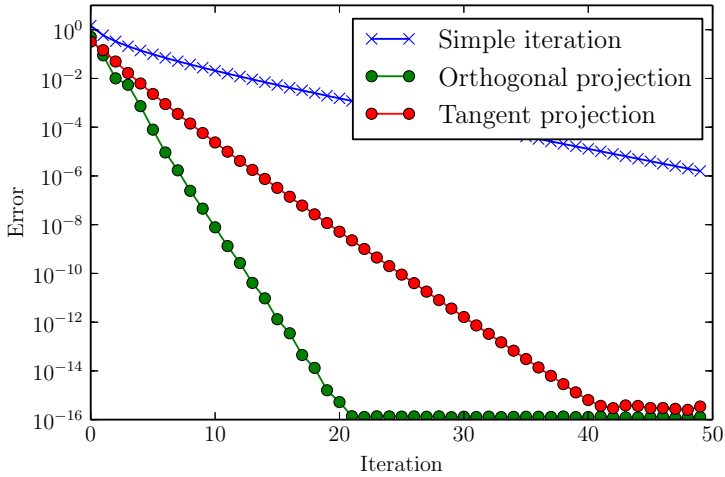}}
\caption{Convergence rates for typical case.}\label{prj:fig11}
\end{figure}
\subsection{Stair case}\label{prj:staircase}
The stair case corresponds to the same $n, m, r$ and exponentially decaying
singular values $\sigma_k = 10^{4-2k}, \quad 1\leq k \leq 7, \quad
\sigma_1 / \sigma_r = 10^{12}$. The results are shown on the Figure \ref{prj:fig2}. The orthogonal component decays quadratically until the next
singular value is achieved.  Meanwhile, the tangent component decays
linearly, and once it hits the same singular value, the orthogonal
component drops again. The steps on the 'stair' correspond to the singular values of $X_*$.

Numerical experiments show that the projector splitting method has ``component-wise''
convergence. Until the first $j$ singular components of the current
point converge to the first $j$
singular components of the fixed point, the last $r-j$ components of
the $X_k$ are ``noisy'' and do not contain useful information.

\begin{figure}[H]
    \centering
    \scalebox{0.7}{
    \input{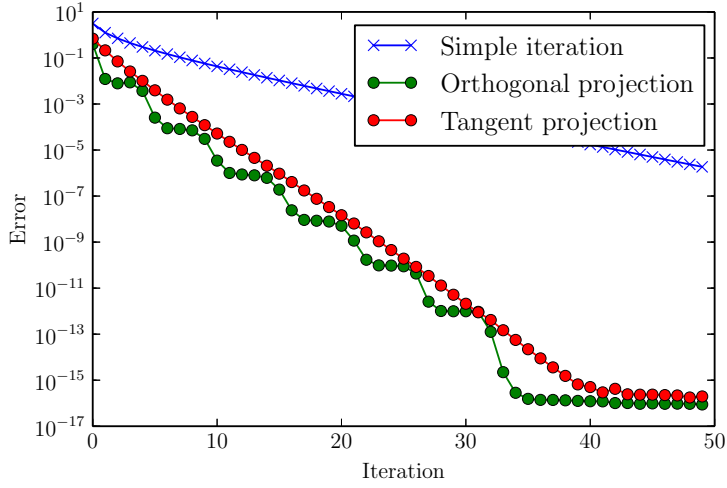}}
\caption{Convergence rates for staircase.}\label{prj:fig2}
\end{figure}
\subsection{Counter-example case}\label{prj:counterexamplecase}
For the following experiment we consider ``nonlinear'' contraction
mapping $$ \Phi(X) = X_* + \delta \Vert X - X_* \Vert X_{\perp},$$
where $X$ is a $2\times 2$ matrix, $X_* = \begin{pmatrix} 1 & 0 \\ 0 &
  0\end{pmatrix},$ $X_{\perp} = \begin{pmatrix} 0 & 0 \\ 0 &
  1\end{pmatrix},$ $\delta = 0.5.$ It shows that original projector
splitting method fails and converges to another stationary
point. Nevertheless this stationary point is unstable and to show that
we introduce a perturbed projector splitting method:
$$Y_{k+1}^{pert} = I(Y_k^{pert}, \Phi(Y_k^{pert}) -Y_k^{pert} + R_k),$$
where $R_k$ is a $n\times m$ matrix with elements taken from the normal
distribution\\
$\mathcal{N}\left(0, \frac{1}{100nm}\Vert\Phi(Y_k^{pert}) -Y_k^{pert} \Vert\right)$.
The convergence is shown at Figure \ref{prj:fig3}: 
\begin{figure}[H]
    \centering
    \scalebox{0.7}{
    \input{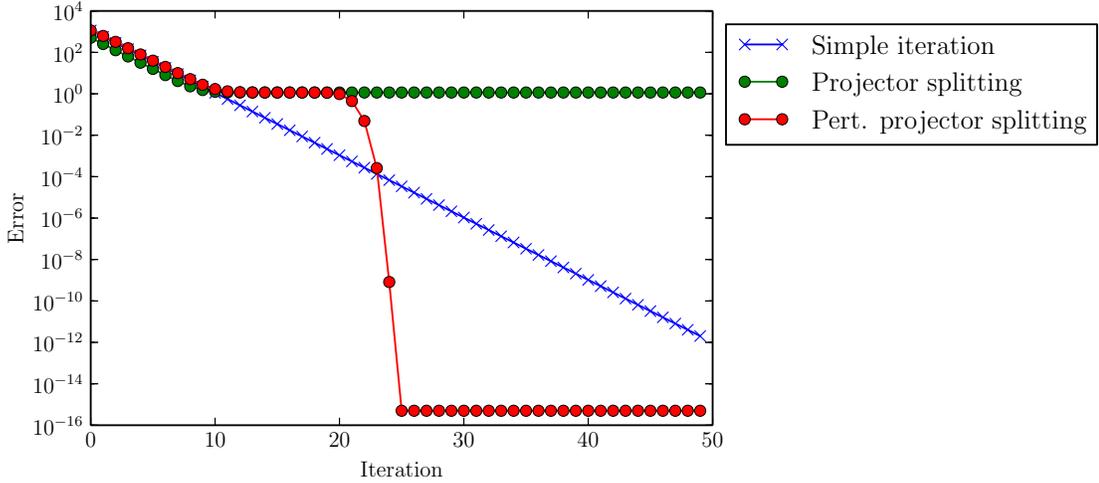}}
\caption{Convergence rates for 'bad functional'}\label{prj:fig3}
\end{figure}

\section{Related work}

Projector splitting method arises naturally as a numerical
integrator for dynamical low-rank approximation of ODE \cite{lubich-timett-2015, lubich-prj-2014} and was
originally proposed in \cite{lubich-lrappr-2007}. 
In this paper we focused on the properties of the projector splitting
method as the retraction onto low-rank
manifold \cite{absil-opt-2009}. It was compared with another retraction
methods in the survey \cite{oseledets-survey-2015}.   

Close results about convergence in the presence of small singular values were obtained in
\cite{lubich-smallsing-2015}. The problem formulation is as follows. Let $X(t)$
be the solution of the ordinary differential equation (ODE):

\begin{align*}
\dot{X}(t) = F(t, X(t)), \quad X(0) = X_0,& \quad X(t) \in \mathbb{R}^{n\times m}, &t\in [0, T], \\
\Vert F(t, X_1) - F(t, X_2)\Vert \leq L,& \quad \forall X_1, X_2 \in \mathbb{R}^{n\times
  m}, &\forall t\in [0, T],\\
\Vert F(t, X) \Vert \leq B,& \quad \forall X \in \mathbb{R}^{n\times
  m}, &\forall t\in [0, T]. \\
\end{align*}
We want to obtain approximation to stationary point $X_*$: $F(t,
X_*) = 0$. We seek for low-rank approximation $Y(t)$ to $X(t)$ and $Y(t)$
satisfies the modified ODE: 
$$\label{prj:modODE}
\dot{Y}(t) = P(Y(t))F(t, Y(t)), \quad Y(0) = Y_0,\quad \rank Y(t) = r,
$$
where $P(Y(t))$ is a projector onto the subspace determined by
$Y(t)$. 
\cite[Theorem 2.1]{lubich-smallsing-2015} states that numerical
approximation $ \widetilde{Y}(t)$ is stable despite the presence of small
singular values of $Y(t)$. However, this result cannot be directly
applied to optimization problems and $F$ should satisfy certain restrictions.

Another close result is a guaranteed local linear convergence for alternating least
squares optimization scheme in convex optimization problems \cite{uschmajew-alsconv-2013}.
Also local convergence results are obtained for modified alternating least squares scheme, such as
maximum block improvement \cite{uschmajew-mbi-2015} and alternating
minimal energy \cite{dolgov-amen-2014}, but for these methods the low-rank
manifold changes at every step.

\section{Conclusions and perspectives}
Our numerical results show that the
staircase is a typical case for linear contraction
mappings. However, conditions of the proved theorem cover only
convergence at the last ``step'' on the stair. 
We plan to formulate conditions for the contraction mapping $\Phi$ for
which ``component-wise'' convergence as for stair case is
guaranteed. Our current hypothesis is that the ``extended'' mapping $\Phi_m(X, X_*)$ should also
satisfy the contraction property for $X_*$. 
 It will be very interesting to explain the nature of the stair
case convergence.

Another important topic for further research is to determine a viable
``a-poste\-riori'' error indicator, since we do not know the orthogonal
component. This will allow to develop rank-adaptive
projector splitting based scheme.
 
The main conclusion of this paper is that projected iterations are
typically as fast as the unprojected ones. We plan to generalize the
paper results for tensor case. 

\section*{Acknowledgements}
 This work was supported by Russian Science Foundation grant 14-11-00659.
We thank Prof. Dr. Christian Lubich and Hanna Walach for fruitful
discussions about projector splitting scheme and retractions on a
low-rank manifold.
We also thank Maxim Rakhuba for his help for improving the manuscript.

\bibliographystyle{siam}
\bibliography{conv_proof}

\begin{thebibliography}{10}

\bibitem{absil-opt-2009}
{\sc P.-A. Absil, R.~Mahony, and R.~Sepulchre}, {\em Optimization algorithms on
  matrix manifolds}, Princeton University Press, 2009.

\bibitem{oseledets-survey-2015}
{\sc P.-A. Absil and I.~V. Oseledets}, {\em Low-rank retractions: a survey and
  new results}, Comput. Optim. Appl., 62 (2015), pp.~5--29.

\bibitem{dolgov-amen-2014}
{\sc S.~V. Dolgov and D.~V. Savostyanov}, {\em Alternating minimal energy
  methods for linear systems in higher dimensions}, SIAM J. Sci. Comput., 36
  (2014), pp.~A2248--A2271.

\bibitem{absil-newton-2009}
{\sc M.~Ishteva, L.~De~Lathauwer, P.-A. Absil, and S.~Van~Huffel}, {\em
  Differential-geometric {N}ewton method for the best rank-($r_1, r_2, r_3$)
  approximation of tensors}, Numerical Algorithms, 51 (2009), pp.~179--194.

\bibitem{lubich-smallsing-2015}
{\sc E.~Kieri, C.~Lubich, and H.~Walach}, {\em Discretized dynamical low-rank
  approximation in the presence of small singular values}, to appear in SIAM J.
  Numer. Anal.,  (2015).

\bibitem{lubich-lrappr-2007}
{\sc O.~Koch and C.~Lubich}, {\em Dynamical low-rank approximation}, SIAM J.
  Matrix Anal. Appl., 29 (2007), pp.~434--454.

\bibitem{uschmajew-mbi-2015}
{\sc Z.~Li, A.~Uschmajew, and S.~Zhang}, {\em On convergence of the maximum
  block improvement method}, SIAM J. Optimiz., 25 (2015), pp.~210--233.

\bibitem{lubich-prj-2014}
{\sc C.~Lubich and I.~V. Oseledets}, {\em A projector-splitting integrator for
  dynamical low-rank approximation}, BIT Numer. Math., 54 (2014), pp.~171--188.

\bibitem{lubich-timett-2015}
{\sc C.~Lubich, I.~V. Oseledets, and B.~Vandereycken}, {\em Time integration of
  tensor trains}, SIAM J. Numer. Anal., 53 (2015), pp.~917--941.

\bibitem{ma-manifold-2011}
{\sc Y.~Ma and Y.~Fu}, {\em Manifold learning theory and applications}, CRC
  press, 2011.

\bibitem{uschmajew-alsconv-2013}
{\sc T.~Rohwedder and A.~Uschmajew}, {\em On local convergence of alternating
  schemes for optimization of convex problems in the tensor train format}, SIAM
  J. Numer. Anal., 51 (2013), pp.~1134--1162.

\bibitem{sato-cg-2015}
{\sc H.~Sato and T.~Iwai}, {\em A new, globally convergent {R}iemannian
  conjugate gradient method}, Optimization, 64 (2015), pp.~1011--1031.

\bibitem{udriste-riemannian-1994}
{\sc C.~Udriste}, {\em Convex functions and optimization methods on Riemannian
  manifolds}, vol.~297, Springer Science \& Business Media, 1994.

\end{thebibliography}
\end{document}